\documentclass{amsart}
\usepackage{amssymb,latexsym}
\usepackage{color}
\theoremstyle{plain}
\newtheorem{theorem}{Theorem}
\newtheorem{corollary}{Corollary}

\newtheorem{lemma}{Lemma}
\theoremstyle{definition}
\newtheorem{definition}{Definition}
\newtheorem{example}{Example}
\newtheorem{notation}{Notation}
\newtheorem{remark}{Remark}

\newtheorem{conjecture}{Conjecture}
\newtheorem{algorithm}{Algorithm}

\newcommand{\red}[1]{{\color{black}#1}}

\date{}

\begin{document}

\title[tangent developable]
{Tensor ranks on tangent developable of Segre varieties}
\author{E. Ballico}
\address{Dept. of Mathematics\\
 University of Trento\\
38123 Povo (TN), Italy}
\email{ballico@science.unitn.it}
\author{A. Bernardi}
\address{Galaad, INRIA 2004 route des Lucioles, BP 93, 06902 Sophia Antipolis, France.}
\email{alessandra.bernardi@inria.fr}
\thanks{The first author was partially supported by MIUR and GNSAGA of INdAM (Italy). The second author
was partially supported by CIRM-FBK (TN-Italy), Marie-Curie FP7-PEOPLE-2009-IEF, INRIA Sophia Antipolis Mediterran\'{e}e Project Galaad (France)
and Mittag-Leffler Institut (Sweden)}
\subjclass{14N05, 14Q05}
\keywords{Secant varieties; tensor rank; tangent developable; Segre Varieties; Comon's conjecture.}

\maketitle

\begin{abstract}
We describe the stratification by tensor rank of the points belonging to the tangent developable of any Segre variety. We give algorithms to compute the rank and a decomposition of a tensor belonging to the secant variety of lines of any Segre variety. We prove Comon's conjecture on the rank of symmetric tensors for those tensors belonging to tangential varieties to Veronese varieties.
\end{abstract}

\section*{Introduction}
In this paper we want to address the problem of tensor decomposition over an algebraically closed field $K$ of characteristic $0$ for tensors belonging to a tangent space of the projective variety that parameterizes completely decomposable tensors.

Let $V_1, \ldots , V_d$ be $K$-vector spaces of dimensions $n_1+1, \ldots , n_d+1$ respectively; the projective variety $X_{n_1,\dots ,n_d}\subset \mathbb{P}(V_1 \otimes \cdots \otimes V_d)$ that parameterizes projective classes of completely decomposable tensors $v_1\otimes \cdots \otimes v_d\in V_1 \otimes \cdots \otimes V_d$ is classically known as a Segre variety (see Definition \ref{segre}). Given a tensor $T\in V_1 \otimes \cdots \otimes V_d$, finding the minimum number of completely decomposable tensors such that $T$ can be written as a linear combination of them (see Definition \ref{rank} for the notion of ``tensor rank") is related to the tensor decomposition problem that nowadays seems to be crucial in many applications like Signal Processing (see eg. \cite{afm}, \cite{nl}, \cite{cst}), Algebraic Statistics (\cite{Mc}, \cite{t}),  Neuroscience (eg. \cite{bclmh}). The specific case of tensors belonging to tangential varieties to Segre varieties  (Notation \ref{tangential}) is studied in \cite{bl} 
and  it turns out to be of certain interest in the context of Computational Biology. In fact in \cite{gs} a particular class of statistical models (namely certain context-specific independence model -- CSI) is shown to be crucial  in machine learning and computational biology. L. Oeding has recently shown in \cite{o} how to interpret the CSI model performed by \cite{gs} in terms of  tangential variety to Segre variety. In this setting  B. Sturmfels and P. Zwiernik in  a very recent paper (\cite{sz}) show  how to derive  parametrizations and implicit equations in cumulants for the tangential variety of the Segre variety $X_{1, \ldots , 1}$ and for certain CSI models (see \cite{brt} for a combinatorial point of view on cumulants).

In this paper, after a preliminary section, we give a complete classification of the tensor rank of an element belonging to the tangent developable of any Segre variety. In particular in Theorem \ref{i1} we will prove that if $P\in T_O(X_{n_1, \ldots , n_d}) $  for certain point $O=(O_1, \ldots , O_d)\in X_{n_1, \ldots , n_d}$, then the minimum number $r$ of completely decomposable tensors $v_{1,i}\otimes \cdots \otimes v_{d,i}\in V_1 \otimes \cdots \otimes V_d$ such that $P=\sum_{i=1}^r [v_{1,i}\otimes \cdots \otimes v_{d,i}]$ is equal to the minimum number $\eta_{X_{n_1, \ldots , n_d}}(P)$ for which  there exist  $E\subseteq \{1,\dots ,d\}$ such that $\sharp (E)=\eta_{X_{n_1, \ldots , n_d}}(P)$ and $T_O(X_{n_1, \ldots , n_d}) \subseteq \langle \cup _{i\in E} Y_{O,i}\rangle$ where $Y_{O,i}$ the $n_i$-dimensional linear subspace obtained by fixing all coordinates $j\in \{1,\dots ,d\}\setminus \{i\}$ equal to $O_j\in \mathbb{P}^n_i$ (see Notation \ref{Y}). 
Such a result was independently proved by J. Buczy\'{n}ski and J. M. Landsberg (see Theorem 7.1 in the second version of \cite{bl}). We propose here a different proof. First of all, the construction that we make in our proof allows to write explicit algorithms for the computation of the rank of a given tensor belonging to the secant variety of lines of any Segre variety (Algorithm \ref{algo1}) and for a decomposition of the same (Algorithm \ref{2}). Moreover in the third and in the fourth versions of \cite{bl}, the authors have removed that result for several months. More recently they resubmitted it in a subsequent paper \cite{bl1} Proposition 1.1.

In Section \ref{algorithms} we give the details for Algorithm \ref{algo1} and for Algorithm \ref{2}. 

In the last section we show how to use  Theorem \ref{i1} in order to prove the so called ``Comon's conjecture" in the particular case in which  the points $P\in \tau(X_{n_1, \ldots  , n_d})$ parameterize symmetric tensors. Let us give more details on that. 
\\
Let $V_1=\cdots =V_d=V$ be a vector space of dimension $n+1$ and consider the subspace $S^dV\subset V^{\otimes d}$ of symmetric tensors. The intersection between the Segre variety $X_{n, \ldots , n}$ and $\mathbb{P}(S^dV)$ is a way to interpret the classical Veronese embedding of $\mathbb{P}^n$  via the sections of the sheaf $\mathcal{O}(d)$. Therefore an element of the Veronese variety $\nu_d(\mathbb{P}^n)=X_{n, \ldots , n}\cap \mathbb{P}(S^dV)$ is the projective class of a completely decomposable symmetric tensor. Now, given a point $P\in \mathbb{P}(S^dV)$ that parameterizes a projective class of a symmetric tensor, we can look at two different decompositions of it. Let $v_{1,i}\otimes \cdots \otimes v_{d,i}\in V^{\otimes d}$ and let $w_{j}^{\otimes d}\in S^dV$ , and ask for the minimum $r$ and the minimum $r'$ such that $P=\sum_{i=1}^r[v_{1,i}\otimes \cdots \otimes v_{d,i}]=\sum_{j=1}^{r'}[w_{j}^{\otimes d}]$. In  2008, at the AIM workshop in Palo Alto, USA (see the report \cite{o1}), P. Comon stated the 
following:
\begin{conjecture}\label{comonconj}\textbf{[Comon's Conjecture]} The minimum integer $r$ such that a symmetric tensor $T\in S^dV$ can be written as 
$$T=\sum_{i=1}^rv_{1,i}\otimes \cdots \otimes v_{d,i}$$
for $v_{1,i}\otimes \cdots \otimes v_{d,i}\in V^{\otimes d}$, $i=1, \ldots , r$, is equal to the minimum integer $r'$ for which there exist $w_{j}^{\otimes d}\in S^dV$, $j=1, \ldots ,r'$ such that 
$$T=\sum_{j=1}^{r'}w_j^{\otimes d}.$$
\end{conjecture}
As far as we know this conjecture is proved if  $r\leq \dim(V)$  (for a general $d$-tensor, $d$ even and large) and if $r=1, 2$ (see \cite{cglm}).
\\ 
In Section \ref{comon} we show that our Theorem \ref{i1} implies that this conjecture is true also for $[T]\in \tau(X_{n, \dots , n})$ (Corollary \ref{comond}).
\\
\\
\textbf{Acknowledgements:} We like to thank B. Sturmfels for asking  to one of us this question at the Mittag-Leffler Institut during the Spring semester 2011 ``Algebraic Geometry with a view towards applications". We also thank the Mittag-Leffler Institut (Djursholm, Stokholm, Sweden), for its hospitality and opportunities.

\section{Preliminaries}
Let us start with the classical definition of the Segre varieties.
\begin{definition}\label{segre}
For all positive integers $d$ and $n_i$, $1\le i \le d$, let 
$$j_{n_1,\dots ,n_d}: \mathbb {P}^{n_1}\times \cdots \times \mathbb {P}^{n_d}\to \mathbb {P}^{N(n_1,\dots ,n_d)},$$ 
with $N(n_1,\dots ,n_d):
= (\prod _{i=1}^{d} (n_i+1))-1$, denote the Segre embedding of $\mathbb {P}^{n_1}\times \cdots \times \mathbb {P}^{n_d}$ obtained by the sections of the sheaf $\mathcal{O}(1,\ldots ,1)$. Set $X_{n_1,\dots ,n_d}:= j_{n_1,\dots ,n_d}(\mathbb {P}^{n_1}\times \cdots \times \mathbb {P}^{n_d})$.
\end{definition}

Observe that if we identify each $\mathbb {P}^{n_i}$ with  $\mathbb {P}(V_i)$ for certain $(n_i+1)$-dimensional vector space $V_i$ for $i=1, \ldots , d$, then an element $[T]\in X_{n_1, \ldots , n_d}$ can be interpreted as the projective class of a completely decomposable tensor $T\in V_1 \otimes \ldots \otimes V_d$, i.e. there exist $v_i\in V_i$ for $i=1, \ldots , d$ such that $T=v_1\otimes \cdots \otimes v_d$.
\\
We can give now the definition of the rank of an element  $P\in \mathbb {P}^{N(n_1,\dots ,n_d)}=\mathbb {P} (V_1 \otimes \cdots \otimes V_d)$.

\begin{definition}\label{rank}
For each $P\in \mathbb {P}^{N(n_1,\dots ,n_d)}$ the rank (or tensor rank) $r_{X_{n_1,\dots ,n_d}}(P)$ of $P$ is the minimal cardinality of a finite set $S\subset X_{n_1,\dots ,n_d}$ such
that $P\in \langle S\rangle$, where $\langle \ \ \rangle$ denote the linear span.
\end{definition}

\begin{notation}\label{tangential}
Let $\tau (X_{n_1,\dots ,n_d})$ denote the tangent developable of $X_{n_1,\dots ,n_d}$, i.e. the union of all tangent spaces $T_PX_{n_1,\dots ,n_d}$ of $X_{n_1,\dots ,n_d}$. Since
$\tau (X_{n_1,\dots ,n_d})$ is closed in the Zariski topology, this is equivalent to the usual definition of the tangent developable of a submanifold of a projective space as the closure of the union of all tangent spaces.
\end{notation}

\begin{remark}\label{ZT}
Fix any $P\in \tau (X_{n_1,\dots ,n_d})\setminus X_{n_1,\dots ,n_d}$ \red{and let 
$J_{2,O}$ 
be the set of pairs $(O,Z)$ such that $O\in  X_{n_1,\dots ,n_d}$ and 
$Z\subset X_{n_1,\dots ,n_d}$  is a zero-dimensional scheme
such that $Z_{red}= \{O\}$, $\deg (Z)=2$.
Then a pair $(O,Z) \in J_{2,O}$ such that $P$ is contained
in the line $\langle Z\rangle$ is almost always unique:
\begin{equation}\label{ZTequal}
P\in \langle Z \rangle \subset T_O X_{n_1,\dots ,n_d}.
\end{equation}}
\end{remark}

\begin{notation}\label{abuse} Let $\tilde{O}=(O_1, \ldots , O_d)\in \mathbb{P}^{n_1}\times \cdots \times \mathbb{P}^{n_d} $.
With an abuse of notation we will write the point $O =j_{n_1 , \ldots , n_d}(\tilde{O})\in X_{n_1, \ldots  n_d}$ as $O= (O_1,\dots ,O_d)$.
\end{notation}

\begin{notation}\label{Y} Fix $O = (O_1,\dots ,O_d)\in X_{n_1, \ldots  n_d}$ as above, we indicate with $Y_{O,i}\subset \mathbb {P}^{N(n_1,\dots ,n_d)}$ the $n_i$-dimensional linear subspace obtained by fixing all coordinates $j\in \{1,\dots ,d\}\setminus \{i\}$ equal to $O_j\in \mathbb{P}^n_i$. To be precise:
$$Y_{O,i} =j_{n_1 , \ldots , n_d} (O_1 , \cdots , O_{i-1}, \mathbb{P}^{n_i}, O_{i+1}, \cdots , O_d).$$
\end{notation}

\begin{remark}\label{YE} Let $Y_{O,i}\subset \mathbb {P}^{N(n_1,\dots ,n_d)}$ the $n_i$-dimensional linear subspace just defined. Observe that, as a scheme-theoretic intersection, we have that:
\begin{equation}
T_OX_{n_1,\dots ,n_d}\cap X_{n_1,\dots ,n_d} = \cup _{i=1}^{d} Y_{O,i}.
\end{equation} 
Moreover, \red{for any triple $(P,O,Z)\in \mathbb {P}^{N(n_1,\dots ,n_d)} \times  J_{2,O}$ 
as in Remark \ref{ZT}, there is a minimal subset $E\subseteq \{1,\dots ,d\}$ such that $\langle Z\rangle \subseteq \langle \cup _{i\in E} Y_{O,i}\rangle$.
We define the \emph{type} $\eta _{X_{n_1,\dots ,n_d}}(P)$ of $P$ as follows:
\begin{equation}\label{eta} 
\eta_{X_{n_1,\dots ,n_d}}(P):= \min_{(O,Z) \in  J_{2,O}} \{ \sharp (E)\, | \, E\subseteq \{1,\dots ,d\}, \, P\in \langle Z\rangle \subseteq \langle \cup _{i\in E} Y_{O,i}\rangle\} .
\end{equation}
}
 Notice that $2\le \eta _{X_{n_1,\dots ,n_d}}(P)\le d$. Moreover for a general $Q\in T_OX_{n_1,\dots ,n_d}$ we have that  $\eta _{X_{n_1,\dots ,n_d}}(Q)=d$.
Furthermore every integer $k\in \{2,\dots ,d\}$ is the type of some point of $\tau (X_{n_1,\dots ,n_d})\setminus X_{n_1,\dots ,n_d}$. Finally for all $Q\in X_{n_1,\dots ,n_d}$ we write $\eta _{X_{n_1,\dots ,n_d}}(Q)=1$
and say that $Q$ has type $1$.
\\
\red{Observe that if $\eta _{X_{n_1,\dots ,n_d}}(P)=2$, then the pair $(O,Z)\in  J_{2,O}$ evincing $\eta _{X_{n_1,\dots ,n_d}}(P)$ as in Remark \ref{ZT} is not unique.}
\end{remark}


In Theorem \ref{i1} we will actually prove that if $P\in \tau (X_{n_1,\dots ,n_d})$, then the integer $ \eta _{X_{n_1,\dots ,n_d}}(P)$ just introduced in (\ref{eta}) is actually the rank of  $P$.
Before proving that theorem we need to introduce the notion of secant varieties and other related objects.

\begin{definition}\label{secant}
For each integer $t\ge 2$ let $\sigma _t(X_{n_1,\dots ,n_d})$ denote the Zariski closure in $\mathbb {P}^{N(n_1,\dots ,n_d)}$ of the union of all $(t-1)$-dimensional linear subspaces of $\mathbb {P}^{N(n_1,\dots ,n_d)}$
spanned by $t$ points of $X_{n_1,\dots ,n_d}$. This object is classically known as the $t$-secant variety of $X_{n_1,\dots ,n_d}$.
\end{definition}

\begin{notation} For each $t\ge 2$ there is a non-empty open subset  of $\sigma _t(X_{n_1,\dots ,n_d})$, that we indicate with  $\sigma _t^{0}(X_{n_1,\dots ,n_d})$, whose elements
 are points of rank exactly equal to $t$. 
 \end{notation}
 
We want to focus our attention on the case $t=2$ that is very particular. 
Theorem \ref{i1} will give the complete stratification by ranks of points in $\sigma _2(X_{n_1,\dots ,n_d})$ (see also 
 \ref{corollary}).
 Indeed, fix $P\in \sigma _2(X_{n_1,\dots ,n_d})$. Obviously $\tau(X_{n_1, \ldots , n_d})\subseteq \sigma_2(X_{n_1, \ldots , n_d})$. If $\tau(X_{n_1,\dots ,n_d})\neq\sigma_2(X_{n_1,\dots ,n_d})$ and if $P\notin \tau (X_{n_1,\dots ,n_d})$,
then $r_{X_{n_1,\dots ,n_d}}(P)=2$ (in fact if $P\in \sigma_2(X_{n_1,\dots ,n_d})\setminus \tau(X_{n_1,\dots ,n_d})$ there exists, by Definition \ref{secant}, two distinct points of  $X_{n_1,\dots ,n_d}$ whose span contains $P$). If $P\in \tau (X_{n_1,\dots ,n_d})$, then in  Theorem \ref{i1} we will show that $r_{X_{n_1,\dots ,n_d}}(P)\in\{1,\dots ,d\}$. 
In particular for each $k\in \{1,\dots ,d\}$, Theorem \ref{i1} will also imply the existence of $P\in \tau (X_{n_1,\dots ,n_d})$ such that $r_{X_{n_1,\dots ,n_d}}(P)=k$.

\begin{definition}
 For any $P\in \mathbb {P}^{N(n_1,\dots ,n_d)}$ the border rank, or border tensor rank, $b_{X_{n_1,\dots ,n_d}}(P)$ is the minimal integer $t$ such that $P\in \sigma _t(X_{n_1,\dots ,n_d})$.
\end{definition}

 Notice that 
 $$b_{X_{n_1,\dots ,n_d}}(P)= 1\; \Longleftrightarrow \; r_{X_{n_1,\dots ,n_d}}(P)= 1 \; \Longleftrightarrow \; P\in X_{n_1,\dots ,n_d}.$$ 
 Thus Theorem \ref{i1} may be considered as the description of the ranks of all points with border rank $2$ (Corollary \ref{corollary}).

For the case of Veronese varieties, i.e. the case of symmetric tensors, and   symmetric border rank $2$ or $3$, see \cite{bgi} and references therein.

\section{Proof of Theorem \ref{i1}.}\label{S2}

This section is entirely devoted to the proof of Theorem \ref{i1}. 

Before going into the details of the proof of Theorem \ref{i1}, we need to remind the following elementary lemma (see e.g. \cite{bb}, Lemma 1).

\begin{lemma}\label{c0}
Fix any $P\in \mathbb {P}^{N(n_1,\dots ,n_d)}$ and two zero-dimensional subschemes $A$, $B$ of $X_{n_1,\dots ,n_d}$ such that
$A \ne B$, $P\in
\langle A\rangle$,
$P\in \langle B\rangle$, $P\notin \langle A'\rangle$ for any $A'\subsetneqq A$ and $P\notin \langle B'\rangle$ for any $B'\subsetneqq B$.
Then $h^1(\mathbb {P}^{N(n_1,\dots ,n_d)},\mathcal {I}_{A\cup B}(1)) >0$.
\end{lemma}

\begin{lemma}\label{c1}
Fix a zero-dimensonal scheme $\tilde{W}\subset \mathbb {P}^{n_1}\times \cdots \times \mathbb {P}^{n_d}$. Then $h^1(\mathbb
{P}^{n_1}\times \cdots \times \mathbb {P}^{n_d},\mathcal {I}_{\tilde{W}}(1,\dots ,1)) =h^1(\mathbb {P}^{N(n_1,\dots ,n_d)},\mathcal
{I}_{j_{n_1,\dots ,n_d}(\tilde{W})}(1))$. 
\end{lemma}

\begin{proof}
It is sufficient to observe that $j_{n_1,\dots ,n_d}$ is the linearly normal embedding induced by the complete linear system
$\vert \mathcal {O}_{\mathbb {P}^{n_1}\times \cdots \times \mathbb {P}^{n_d}}(1,\dots ,1)\vert$
 and that
$h^1(\mathbb
{P}^{n_1}\times \cdots \times \mathbb {P}^{n_d},\mathcal {O}_{\mathbb {P}^{n_1}\times \cdots \times \mathbb {P}^{n_d}}(1,\dots
,1))=0$. 
\end{proof}

We are now ready to prove Theorem \ref{i1}.

\begin{theorem}\label{i1} Let $\tau(X_{n_1, \ldots , n_d})$ be the tangential variety   of the Segre variety $X_{n_1, \ldots , n_d}$.
For each $P \in  \tau (X_{n_1,\dots ,n_d})$ we have  that the tensor rank of $P$ is:
$$r_{X_{n_1,\dots ,n_d}}(P) =\eta _{X_{n_1,\dots ,n_d}}(P)$$
where the integer $\eta _{X_{n_1,\dots ,n_d}}(P)$ is the type of $P$ defined in (\ref{eta}).
\end{theorem}

We write here the strategy of the proof in order to help the reader in following it.

First of all, we observe that if  $\eta _{X_{n_1,\dots ,n_d}}(P)=1$ there is nothing to prove. So we assume that there exist a point $O\in X_{n_1,\dots ,n_d}$ such that $P\in T_O(X_{n_1,\dots ,n_d})\setminus \{ O\}$.

Moreover we point out that the inequality $r_{X_{n_1,\dots ,n_d}}(P) \le \eta _{X_{n_1,\dots ,n_d}}(P)$ (see (\ref{obvious})) is obvious, then we need only to prove the reverse inequality. 

Then we split the proof in the following cases:

(a) If all the $n_i=1$ and $\eta _{X_{1,\dots ,1}}(P)=d$ then $ r_{X_{1,\dots ,1}}(P)=d$. We will give a proof by contradiction: we assume that $\eta _{X_{1,\dots ,1}}(P)=d$ and that  $r_{X_{1,\dots ,1}}(P)<d$ and we show that in each of the following sub-cases we get a contradiction:

\indent \indent (a1) $O\notin \langle S\rangle$, where $S$ is the set of points computing the rank of $P$;

\indent \indent (a2) $O\in  S$;

\indent \indent (a3) $O\notin S$ and $O\in \langle S\rangle$.

(b) If all the $n_i=1$ and $\eta _{X_{1,\dots ,1}}(P)<d$ $\Rightarrow r_{X_{1,\dots ,1}}(P)=\eta _{X_{1,\dots ,1}}(P)$. 

(c) We conclude the proof by showing that the theorem is true for all $n_i\geq 2$  (this part may be bypassed quoting \cite{lw} where it is shown that secant variety of lines of a Segre variety is contained in the subspace variety).

\begin{proof}
Fix $P\in \tau (X_{n_1,\dots ,n_d})$ and look for $r_{X_{n_1,\dots ,n_d}}(P)$.
\\ 
Since $\eta _{X_{n_1,\dots ,n_d}}(P) = 1$ $\Longleftrightarrow $ $P\in X_{n_1,\dots ,n_d}$ $\Longleftrightarrow $ $r_{X_{n_1,\dots ,n_d}}(P) =1$,
the case $P\in X_{n_1,\dots ,n_d}$ is obvious. 
Hence we may assume $P\notin X_{n_1,\dots ,n_d}$. 
Take $(O,Z)\in  J_{2,O}$ evincing $\eta _{X_{n_1,\dots ,n_d}}(P)$.
Moreover we can think of $Z\subset X_{n_1, \ldots , n_d}$ as 
$$Z= j_{n_1,\dots ,n_d}(\widetilde{Z})$$ 
with $\widetilde{Z} \subset \mathbb {P}^{n_1}\times
\cdots \times \mathbb {P}^{n_d}$ and $\widetilde{Z}\cong Z$.
\\
Now, as in Remark \ref{YE}, fix $E\subseteq \{1,\dots ,d\}$  such that 
$$\sharp (E) =\eta _{X_{n_1,\dots ,n_d}}(P)$$
 and 
$$P\in \langle \cup_{i\in E} Y_{O,i}\rangle$$
(where $Y_{O,i}$ are defined as in Notation \ref{Y}).
\\
 Since each $Y_{O,i} \subset \mathbb{P}^{n_1, \ldots , n_d}$ is a linear subspace, then for each $i\in E$ there is $Q_i\in Y_{O,i}$ such that $P\in \langle \cup _{i\in E}Q_i\rangle$. Thus 
 \begin{equation}\label{obvious}r_{X_{n_1,\dots ,n_d}}(P) \le \eta _{X_{n_1,\dots ,n_d}}(P).\end{equation}
Therefore we need simply  to prove
the opposite inequality. 

For each $j\in \{1,\dots ,d\}$ and each $Q_j\in \mathbb {P}^{n_j}$ (or, with the same abuse of notation as in Notation \ref{abuse}, we can think at a point $Q$ in the Segre variety obtained as $j_{n_1 \ldots , n_d}(\tilde{Q})$ with $\tilde{Q}=(Q_1 , \ldots , Q_d)\in \mathbb{P}^{n_1}\times \cdots \times \mathbb{P}^{n_d}$ and then write  $Q = (Q_1,\dots ,Q_d)\in X_{n_1,\dots ,n_d}$), set:
\begin{equation}\label{X(qj)}
X_{n_1,\dots ,n_d}(Q_j,j):= \{(A_1,\dots ,A_d)\in X_{n_1,\dots ,n_d}:A_j=Q_j\}. 
\end{equation}
 Hence $X(Q_j,j)$ is an $(n_1+\cdots
+n_d-n_j)$-dimensional product of $d-1$ projective spaces embedded as a Segre variety in a linear subspace of $\mathbb
{P}^{N(n_1,\dots ,n_d)}$. 

Now our proof splits in two parts: in the first one ((a) together with (b)) we study the case of the Segre product of $d$ copies of $\mathbb{P}^1$'s (i.e. we prove the theorem for $\tau(X_{1, \ldots ,1})$); in part (c) we generalize the result obtained for  $X_{1, \ldots ,1}$ to the general case $X_{n_1, \ldots ,n_d}$ with $n_i\geq 1$, $i=1, \ldots , d$.
\\
\\
\quad (a) Here we assume $n_i=1$ for all $i$ and $\eta _{X_{n_1,\dots ,n_d}}(P)=d$. 
Assume $r:= r_{X_{1,\dots ,1}}(P) < d$
and fix a 0-dimensional scheme $S\subset X_{n_1, \ldots , n_d}$ that computes the rank $r$ of $P$, i.e. fix 
$$\tilde{S} \subset \mathbb {P}^1\times \cdots \times \mathbb {P}^1$$ 
such that 
$$j_{n_1,\dots ,n_d}(\tilde{S})=S, \; P\in \langle S\rangle  \hbox{ and }\sharp (j_{n_1,\dots ,n_d}(S))=r.$$ Write 
$$S = \{Q_1,\dots ,Q_r\}$$ and let $(Q_{i,1},\dots ,Q_{i,d})$ be the components of each $Q_i\in X_{1, \ldots ,1}$ with $i=1\ldots ,r$, i.e. let $\tilde{Q}_i=(Q_{i,1},\dots ,Q_{i,d})\in \mathbb{P}^1 \times \cdots \times \mathbb{P}^1$ s.t. $j_{n_1, \ldots , n_d}(\tilde{Q}_i)=Q_i$ and then, according with Notation \ref{abuse}, write $Q_i=(Q_{i,1},\dots ,Q_{i,d})$.
\\
Now write 
$$\tilde{O} =(O_1,\dots ,O_d)
\in \mathbb {P}^1\times \cdots \times \mathbb {P}^1$$ and 
$$O=j_{n_1, \ldots , n_d}(\tilde{O}).$$ Choose homogeneous coordinates on $\mathbb {P}^1$.
Since $X_{1,\dots ,1}$ is a homogeneous variety, it is sufficient to prove the case
$O_i =[1,0]$ for all $i=1, \ldots , d$. 
\\
Notice that $\deg (Z\cup S) = r+2$ if $O\notin S$ and $\deg (Z\cup S)=r+1$ if $O\in S$.
\\
Since $S$ computes $r_{X_{1,\dots ,1
}}(P)$, we have $P\notin \langle  j_{n_1,\dots ,n_d}(\tilde{S}')\rangle$ for any $\tilde{S}'\subseteq \tilde{S}$. Since $P\ne O$ and $\{O\}$ is the only proper subscheme
of $Z$, 
we have $P\notin \langle Z'\rangle$ for all proper subschemes $Z'$ of $Z$. Since $P\in \langle Z\rangle \cap \langle S\rangle$, then, by Lemma \ref{c0},
we have $h^1(\mathcal {I}_{ S\cup Z}(1)) >0$. Thus to get a contradiction and prove Theorem \ref{i1} in the case
$n_i=1$ for all $i=1, \ldots , d$ and $\eta _{X_{1,\dots ,1}}(P)=d$, it is sufficient to prove $h^1(\mathcal {I}_{S\cup Z}(1))=0$,
i.e. $h^1(\mathbb
{P}^{1}\times \cdots \times \mathbb {P}^{1},\mathcal {I}_{\widetilde{Z}\cup \tilde{S}}(1,\dots ,1)) =0$ where, as above,  $\widetilde{Z} \subset \mathbb {P}^{n_1}\times
\cdots \times \mathbb {P}^{n_d}$  s.t. $Z= j_{n_1,\dots ,n_d}(\widetilde{Z})$.

First assume the existence of an integer $j\in \{1,\dots ,d\}$ such that $Q_{i,j}=[1,0]$ for all $i\in \{1,\dots ,r\}$.
 We get $S\subset X_{1,\dots ,1}([1,0],j)$, where $X_{n_1, \ldots , n_d}(Q_j,j)$ is defined in (\ref{X(qj)}).
Hence $P\in \langle X_{1,\dots ,1}([1,0],j)\rangle$. However $T_OX_{1,\dots ,1}\cap X_j = \langle \cup _{i\ne j}Y_{O,i}\rangle$. Hence $\eta (P)\le d-1$, but this is a contradiction. Thus:
$$
\hbox{for each } j\in \{1,\dots ,d\} \hbox{ there is } Q_{i_j}\in S \hbox{ such that } Q_{i_j,j}\ne [1,0].
$$

\quad (a1) 
 Here we assume $O\notin \langle S\rangle$. 
Since $S$ computes $r_{X_{1,\dots
,1}}(P)$,
it is linearly independent, i.e. (by Lemma \ref{c1}) $h^1(\mathbb {P}^1\times \cdots \times \mathbb {P}^1,\mathcal {I}_{\tilde{S}}(1,\dots ,1))=0$. 
\\
Since
$O\notin
\langle S\rangle$, we get that
$\tilde{S}\cup
\{\tilde{O}\}$ is linearly independent, i.e. $h^1(\mathbb {P}^1\times \cdots \times \mathbb {P}^1,\mathcal {I}_{\tilde{S}\cup \{\tilde{O}\}}(1,\dots
,1))=0$. 
\\
We fix
$i\in
\{1,\dots ,r\}$ such that $Q_{i,1} \ne [1,0]$ (we just saw the existence of such an integer $i$). 
\\
Write $S_1:= S\cap X_{1,\dots
,1}(Q_{i,1},1)$, where $X_{n_1, \ldots , n_d}(Q_j,j)$ is defined in (\ref{X(qj)}). By construction $Q_i\in S_1$ and hence
$\sharp (S_1)\ge 1$.
\\
Assume for now that $S_1\ne S$ and that there exist $j\in S\setminus S_1$ such that $Q_{j,2}\ne [1,0]$.
Set $S_2:= S\cap X_{1,\dots ,1}(Q_{j,2})$. And so on constructing subsets $S_1,\dots ,S_j$ of $S$ such that:
\begin{itemize}
\item $S_j \nsubseteq \cup _{1\le i<j} S_i$,
\item $Q_{k,i}\ne [1,0] \hbox{ for all } k\in S_i$,
\item $ S_i=S\cap X_{1,\dots ,1}(Q_{h,i},i)$ for all $h\in S_i$,
\end{itemize}
 until we arrive at one of the
following cases:
\begin{itemize}
\item[(i)] $S_1\cup \cdots \cup S_j = S$;
\item[(ii)] $S_1\cup \cdots \cup S_j \ne S$ and $Q_{k,j+1} = [1,0]$ for all $k\in S\setminus (S_1\cup \cdots \cup S_j)$.
\end{itemize}

Now fix an index $m_{i+1}\in S_{i+1}\setminus S_i$, $1 \le i \le j-1$, and set 
$$D_i:= X_{1,\dots ,1}(Q_{m_i,i},i), \; 1\le i \le j,$$ 
i.e. according with (\ref{X(qj)}), $D_i:= \{(A_1,\dots ,A_d)\in X_{n_1,\dots ,n_d}:A_i=Q_{m_i} \hbox{ with } m_i\in S_i\setminus S_{i-1}\}$ for $1\le i \le j$.

First assume that (i) occurs (with $j$ minimal).
Fix $B_i\in \mathbb {P}^1\setminus \{[1,0]\}$, $j+1 \le i \le d-1$ and set:
\begin{itemize}
\item $D_i:= X_{1,\dots ,1}(B_i,i)$, if $j+1\le i \le d-1$;
\item $D_d:= X_{1,\dots ,1}(O_d,d)$;
\item $D:= \cup _{i=1}^{d} D_i$.
\end{itemize}
 Notice that obviously $D\in \vert \mathcal {O}_{X_{1,\dots ,1}}(1)\vert$ and also  that $S\cup \{O\}\subset D$. 
Moreover observe  that $O\in D_i$ if and only if $i=d$. 
Finally, $D_d$ is smooth
at $O$ and $T_OD$ is spanned by $\cup _{i=1}^{d-1} X_{1,\dots ,1}(O_1,i)$.
Therefore $Z\nsubseteq D$ and
 $Z\cup S$ imposes one more condition to $\vert \mathcal {O}_{\mathbb {P}^1\times \cdots \times \mathbb {P}^1}(1,\dots ,1)\vert$ than $S\cup \{O\}$. Since $j_{1,\dots ,1}(\tilde{S}\cup \{\tilde{O}\})$ is linearly independent,
we get $h^1(\mathbb
{P}^{1}\times \cdots \times \mathbb {P}^{1},\mathcal {I}_{\widetilde{Z}\cup \tilde{S}}(1,\dots ,1))=0$ that is a contradiction.

Now assume that (ii) occurs and set:
\begin{itemize} 
\item $M_{j+1} = X_{1,\dots ,1}([1,0],j+1)$;
\item $M_h:=  X_{1,\dots ,1}([1,0],h)$,  for all $h\in \{j+2,\dots d\}$;
\item  $D':= \bigcup _{i=1}^{j} D_i\cup \bigcup _{h=j+1}^{d} M_h$.
\end{itemize}
 Notice that $D'\in \vert \mathcal {O}_{X_{1,\dots ,1}}(1)\vert$ and that $S\cup \{O\}\subset D$.
 The hypersurface $M_{j+1}$ is the unique irreducible component of $D'$ containing $O$. 
Since $M_{j+1}$ is smooth
at $O$ and $T_OM_{j+1}$ is spanned by $\cup _{i\ne j+1}^{d-1} X_{1,\dots ,1}(O_1,i)$, we get as above that $h^1(\mathbb
{P}^{1}\times \cdots \times \mathbb {P}^{1},\mathcal {I}_{\widetilde{Z}\cup \tilde{S}}(1,\dots ,1)) =0$, and than another contradiction.

\quad (a2) Here we assume $O\in S$. 
Hence $S\cup \{O\} =S$ and $j_{1,\dots ,1}(\tilde{S}\cup \{\tilde{O}\})$ is linearly independent. Set $S':= S\setminus \{O\}$. We make the construction of step (a1) with $S'$ instead of $S$, defining the subsets $S_i$ of $S'$ until we get an integer
$j$ such that either $S' = S_1\cup \cdots \cup S_j$ or $S_1\cup \cdots \cup S_j \ne S'$ and $Q_{j+1,i} = [1,0]$ for all $i\in S'\setminus (S_1\cup \cdots \cup S_j)$. In both cases we add the other $d-j$
hypersurfaces, exactly one of them containing $O$. Since $\deg (Z\cup S) = \deg (S\cup \{O\})+1$, we get $h^1(\mathbb
{P}^{1}\times \cdots \times \mathbb {P}^{1},\mathcal {I}_{\widetilde{Z}\cup \tilde{S}}(1,\dots ,1))=0$ as in step (a1) and hence we get a contradiction.

\quad (a3) Here assume $O\notin S$ and $O\in \langle S\rangle$. 
 Hence $\langle Z\rangle \subset \langle S\rangle$. Thus there is $S'\subset S$ such that $\sharp (S') =\sharp (S)-1$ and $\langle S'\cup \{O\}\rangle = \langle S\rangle$.
Hence the set $S_1:= S'\cup \{O\}$ computes $r_{X_{1,\dots ,1}}(P)$. Apply step (a2) to the set $S_1$.
\\
\\
\quad (b) Here we assume $n_i=1$ for all $i$ and $r:= \eta _{X_{1,\dots ,1}}(P)<d$. 
Let $E\subset \{1,\dots ,d\}$ be the minimal subset
such that $P\in \langle \cup _{i\in E} Y_{O,i}\rangle$. By the definition of the type $\eta _{X_{1,\dots ,1}}(P)$ of $P$ we have
$\sharp (E) = \eta _{X_{1,\dots ,1}}(P)$. Set $X':= \{(U_1,\dots ,U_d)\in X_{1,\dots ,1}: U_i=[1,0]$ for all $i\notin E\}$.  We identify
$X'$ with a Segre product of $r$ copies of $\mathbb {P}^1$. Obviously $\eta _{X'}(P) = \eta _{X_{1,\dots ,1}}(P)$. By step (a) we have $r_{X'}(P) =\eta _{X'}(P)$.  We
have $ r_{X_{1,\dots ,1}}(P)=r_{X'}(P) $ by the concision property of tensors (\cite{bl}, Corollary 2.2, or \cite{l}, Proposition 3.1.3.1).
\\
\\
\quad (c) Here we assume $n_i\ge 2$ for some $i$. 
Since $P\in \langle \cup _{i=1}^{d} Y_{O,i}\rangle $, there is
$U_i\in Y_{O,i}$ such that $P\in \langle \{U_1,\dots ,U_d\}\rangle$. Let $U^i_i\in \mathbb {P}^{n_i}$ be the $i$-th component of $U_i$.
The line $L_i\subseteq \mathbb {P}^{n_i}$ is the line spanned by $O_i$ and $U^i_i$. We have $P\in \langle \prod _{i=1}^{d} L_i\rangle$
and  $\eta _{X_{n_1,\dots ,n_d}}(P) = \eta _{X_{1,\dots ,1}}(P)$, where we identify $j_{n_1,\dots ,n_d}(\prod _{i=1}^{d} L_i)$ with the Segre variety
$X_{1,\dots ,1}$. By parts (a) and (b) we have $ r _{X_{1,\dots ,1}}(P)= \eta _{X_{1,\dots ,1}}(P)$. We
have $r_{X_{n_1,\dots ,n_d}}(P) = r_{X_{1,\dots ,1}}(P)$ by the concision property of tensors (\cite{bl}, Corollary 2.2, or \cite{l}, Proposition 3.1.3.1).
\end{proof}

\begin{corollary}\label{corollary} Let $P\in \sigma_2(X_{n_1, \ldots , n_d})$, then:
\begin{itemize}
\item $r_{X_{n_1, \ldots , n_d}}(P)=1$ iff $P\in X_{n_1, \ldots , n_d}$;
\item $r_{X_{n_1, \ldots , n_d}}(P)=2$ iff either  $P\in \sigma_2(X_{n_1, \ldots , n_d})\setminus \tau(X_{n_1, \ldots , n_d})$ or there exist $O\in X_{n_1, \ldots , n_d}$, $O\neq P$, and $Y_{O,i} , Y_{O,j}\subset \mathbb{P}^{N(n_1, \ldots , n_d)f}$ as in Notation \ref{Y}, such that $P\in T_{O}(X_{n_1, \ldots , n_k})\subset Y_{O,i}\cup Y_{O,j}$ for certain $i\neq j \in \{1, \ldots ,d\}$;
\item $r_{X_{n_1, \ldots , n_d}}(P)=k$ with $3\leq k \leq d$ iff $k$ is the minimum integer s.t. there exist $Y_{O,i_1} , \ldots  ,Y_{O,i_k} \subset \mathbb{P}^{N(n_1, \ldots , n_d)}$ as in Notation \ref{Y}, such that $P\in T_{O}(X_{n_1, \ldots , n_k})\subset \cup_{j=1, \ldots k} Y_{O,i_j}$ for certain $i_j \in \{1, \ldots ,d\}$, $j=1, \ldots ,k$.
\end{itemize}
\end{corollary}

\begin{proof} This corollary follows straightforward from Theorem \ref{i1} and the fact that $\sigma_2(X_{n_1, \ldots , n_d})\setminus \tau(X_{n_1, \ldots , n_d})=\sigma_2^0(X_{n_1, \ldots , n_d})$ when it is not empty.
\end{proof}

The three cases of this Corollary actually occur and can be deduced from the proof of Theorem \ref{i1}.

\begin{example} Let us write for convenience $\mathbb{P}^{n_i}=\mathbb{P}(V_i)$ for certain $(n_i+1)$-dimensional vector spaces over $K$.

\begin{itemize}

\item The points $P\in \mathbb{P}(V_1 \otimes \cdots \otimes V_d)$ for which there exist $v_i\in V_i$, for $i=1, \ldots , d$, such that $P=[v_1 \otimes \cdots \otimes v_d]$, have $r_{X_{n_1, \ldots , n_d}}(P)=1$.

\item Let $P_1=[v_{1,1} \otimes \cdots \otimes v_{1,d}],P_2=[v_{2,1} \otimes \cdots \otimes v_{2,d}]\in X_{n_1, \ldots , n_d}$ with $v_{1,1} \otimes \cdots \otimes v_{1,d}, v_{2,1} \otimes \cdots \otimes v_{2,d}\in V_1 \otimes \cdots \otimes V_d$ linearly independent, then $P=\lambda_1 P_1 + \lambda_2 P_2$, for non-zero coefficients $\lambda_1, \lambda_2 \in K$, has $r_{X_{n_1, \ldots , n_d}}(P)=2$.

\item 
We can observe that, for any $r\leq d$, with an abuse of notation,  there is an obvious way to see $V_1 \otimes \cdots \otimes V_r$ as a natural subspace of $V_1 \otimes \cdots \otimes V_d$. Roughly speaking this is the same to say that  the Segre variety of $r$ factors can be seen as a subvariety of the Segre variety of $d$ factors. 
Let $O=[w_1 \otimes \cdots \otimes w_r] \in X_{n_1, \ldots , n_r}\subseteq  X_{n_1, \ldots , n_d}$.
Take $v_i\in V_i$, $i=1, \ldots , r$, such that $\{ w_1, \ldots , v_i , \ldots , w_r \}$ are linearly independent, then $P=\lambda_1 [v_1 \otimes w_2 \otimes \cdots \otimes w_r] + \cdots + \lambda_r[w_1 \otimes \cdots \otimes w_{r-1} \otimes v_r]$ has rank $r$ certain non zero  $\lambda_1, \ldots, \lambda_r \in K$. 
\end{itemize}

\end{example}

\section{Algorithms}\label{algorithms}

The proof of Theorem \ref{i1} turns out to be useful to produce an algorithm to compute the rank of tensors of border rank 2 (Algorithm \ref{algo1}) and also an algorithm to find one of its decompositions (Algorithm \ref{algo2}). 

Let us spend few lines for a more precise but brief discussion on what is know about uniqueness of tensor decomposition in our case. First of all let us observe that since we are studying the case of tensors of border rank 2, we can use the so called ``concision property of tensors" (\cite{bl}, Corollary 2.2, or \cite{l}, Proposition 3.1.3.1) to claim that all the possible decompositions of a tensor $P\in V_1\otimes \cdots \otimes V_d$ are of type $P=\sum_{i_1}^r=u_{i,1} \otimes \cdots \otimes u_{i,d}$ with $u_{i,j}\in U_j \subseteq V_j$,  $\dim U_j=2$, for $i=1, \ldots, r$, $j=1, \ldots, d$. Even more, if $r:= \eta_{X_{n_1, \ldots , n_d}} (P) <d$, then our problem reduces to the case in which $d-r$ of the factors have dimension $0$. Hence this reduces to a case of Segre variety  $X_{1, \ldots , 1}= X_{1^r}$ of $r$ copies of $\mathbb{P}^1$, and rank $r$.
If a border rank 2 tensor  is symmetric, then the uniqueness of the decomposition is known to hold for any $d$ (this is the Sylvester case  \cite{Sy}, \cite{CS}, \cite{bgi}, \cite{bcmt} together with \cite{l}, Exercise 3.2.2.2 for the concision property in the symmetric case). The case in which $b_{X_{n_1,\dots ,n_d}}(P)= 2 < r_{X_{n_1,\dots ,n_d}}(P)=r$ is completely understood if $P$ is symmetric (\cite{Sy}, \cite{CS}, \cite{bgi}, \cite{bcmt}): for any tensor on a tangent line to a rational normal curve of degree $r$ there is an $r-1$ dimensional family of possible decompositions. The nonsymmetric case is analogous: Let  $P\in T_O(X_{n_1,\dots ,n_d})$ for an element $O\in X_{n_1,\dots ,n_d}$, and $r:= \eta_{X_{n_
1, \ldots , n_d}}  (P)$. As above, by the concision property,  our problem reduces to the case $P\in T_O(X_{1^r})$ and rank of $P$ equal to $r$.
We have a ``~framing~'' of the $r$-dimensional linear space $T_O(X_{1^r})$ formed
by the $r$-lines $L_1,\dots ,L_r$ through $O$ whose union is $(X_{1^r})\cap T_O(X_{1^r})$. Since
$\eta_{X_{n_1, \ldots , n_d}} (P)=r$, $P$ is not in the linear span of $r-1$ of these lines. Fix any $P_i\in L_i\setminus \{O\}$,
$1\le i \le r-1$ and set $M:= \langle P_1,\dots ,P_{r-1}\rangle$. For each $i\ge 2$, we have $L_i\nsubseteq \langle
L_1\cup \cdots \cup L_i\rangle$.  Hence $\dim (M)=r-2$.

\quad {\emph {Claim:}} There is a unique $P_r\in L_r$ such that $P\in \langle P_1,\dots ,P_r \rangle$.

\quad {\emph {Proof of the Claim:}} Since
$\langle O,P_1,\dots ,P_{r-1}\rangle = \langle L_1\cup \cdots \cup L_{r-1}\rangle$,
$\dim (M) = r-2$ and $T_O(X_{1^r})= \langle L_1\cup \cdots \cup L_r\rangle$,
we have $T_O(X_{1^r}) = \langle L_r\cup M\rangle$ and $L_r\cap M =\emptyset$.
Hence there is $P_r\in L_r$ such that $P\in \langle P_1,\dots ,P_r\rangle$.
Since $\eta_{X_{n_1, \ldots , n_d}} (P)=r$, we have $P\notin \langle L_1\cup \cdots \cup L_{r-1}\rangle$.
Hence $P_r\ne O$. Assume that $P_r$ is not unique and call $P'_r\in L_r$ another point
such that $P\in \langle P'_r\cup M\rangle$. The line $L_r = \langle \{P'_r,P_r\}\rangle$ would
be contained in the $(r-1)$-dimensional linear space $\langle \{P\}\cup M\rangle$.
Hence we would have $L_r\cap M \ne \emptyset$, a contradiction. This proves the following:

\begin{remark} If  $P\in T_O(X_{n_1,\dots ,n_d})$ for an $O\in X_{n_1,\dots ,n_d}$, and $r:= \eta_{X_{n_1, \ldots , n_d}} (P)$, then   the number of solutions of  the tensor decomposition of $P$  depends \red{ on at least $r-1$ parameters and exactly $r-1$ parameters each time the
pair $(O,Z)$ is uniquely determined by $P$.}
\end{remark}

We need now to introduce the notion of flattening and the definition of Hankel operator.

\begin{definition} Let $V_1,\ldots ,V_d$ be vector spaces of dimensions $n_1+1, \ldots , n_d+1$ respectively. Let $(J_1, J_2)$ be a partition of the set $\{1,\ldots ,d\}$. If $J_1 = \{h_1,\ldots ,h_s\}$ and $J_2 = \{1,\ldots ,d\} \setminus J_1 = \{k_1,\ldots ,k_{d-s}\}$, the $(J_1,J_2)$-flattening of $V_1 \otimes \cdots \otimes V_d$ is the following:
$$V_{J_1}\otimes V_{J_2} =(V_{h_1}\otimes \cdots \otimes V_{h_s})\otimes \cdots \otimes (V_{k_1}\otimes \cdots \otimes V_{k_{d-s}}).$$
\end{definition}

\begin{definition}
Let $n := \sum_{i=1}^d n_i$ , set $R := K[x_1,\ldots ,x_n]$. For any $\Lambda \in R^*$, we define the Hankel operator $H_{\Lambda}$ as 
$H_{\Lambda} : R\rightarrow R^*$, $p\mapsto p\cdot {\Lambda}$ where $p\cdot {\Lambda}$ is the linear operator 
$p\cdot {\Lambda} : R\rightarrow K$, $q\mapsto  {\Lambda}(pq)$.
\end{definition}

\begin{algorithm}[\textbf{Rank of a border rank 2 tensor}]\label{algo1} $\hbox{ }$\\
\textbf{Input:} A tensor $T\in V_1 \otimes \cdots \otimes V_d$, with $V_1,\ldots , V_d$ vector spaces of dimensions $n_1+1, \ldots , n_d+1$ respectively.
\\
\textbf{Output:} Either $T\notin \sigma_2(X_{n_1, \ldots , n_d})$, or the rank of $T$.
\begin{enumerate}
	\item\label{1} Write $T$ as an element of $V_{J_1}\otimes V_{J_2}$  for any $(J_1,J_2)$-flattening of $V_1 \otimes \cdots \otimes V_d$.
	\item\label{2} 
Compute all the $2\times 2$ minors of $V_{J_1}\otimes V_{J_2}$  for any $(J_1,J_2)$-flattening of $V_1 \otimes \cdots \otimes V_d$. If all of them are equal to 0, then $r(T)=1$ (see e.g. \cite{Ha}), otherwise go to Step (\ref{3}).
	\item\label{3} Compute all the $3\times 3$ minors of $V_{J_1}\otimes V_{J_2}$  for any $(J_1,J_2)$-flattening of $V_1 \otimes \cdots \otimes V_d$. If at least one of them is different from 0, then $T\notin \sigma_2(X_{n_1, \ldots , n_d})$ and this algorithm stops here; otherwise  $T\in \sigma_2(X_{n_1, \ldots , n_d})$ (see \cite{lm}) and go to Step (\ref{4}).
	\item\label{4} Find $\Lambda \in (V_1 \otimes \cdots \otimes V_d)^*$ that extends $T^*$  (for a precise definition of extension see \cite{bbcm}) and such that $rk(H_{\Lambda})=2$ then pass to Step (\ref{5}).
	\item\label{5} Compute the roots of $ker H_{\Lambda}$ by generalized eigenvector computation (see \cite{bgl}) and check if the eigenspaces are simple. If yes then the rank of $T$ is 2 (see \cite{bbcm}), otherwise go to Step (\ref{6}).
\item\label{6} Write $T$ as a multilinear polynomial $t$ in the ring $K[x_{1,0}, \ldots , x_{1,n_1}; \ldots \ldots ; x_{d,0}, \ldots , x_{d,n_d} ]$, then pass to Step (\ref{7}).
	\item\label{7} Use \cite{car} to write $t$ in the minimum number $q$ of variables. Then the rank of $t$ is equal to $q/2$ (in fact, from the proof of Theorem \ref{i1}, it is always possible to write $T$ as an element of $\tau(X_{1, \ldots , 1})$, then its representative polynomial will be a multilinear form in $K[l_{1,0}, l_{1,1}; \ldots ; l_{q,0}, l_{q,1}]$ with $l_{i,0}, l_{i,1}$ linear forms in $K[x_{i,0}, \ldots , x_{i,n_i}]$ for $i=1, \ldots ,q$).
\end{enumerate}

\end{algorithm}

\begin{algorithm}[\textbf{Decomposition of a border rank 2 tensor}]\label{algo2} $\hbox{ }$\\
\textbf{Input:} A tensor $T\in V_1 \otimes \cdots \otimes V_d$, with $V_1,\ldots , V_d$ vector spaces of dimensions $n_1+1, \ldots , n_d+1$ respectively.
\\
\textbf{Output:} Either $T\notin \sigma_2(X_{n_1, \ldots , n_d})$, or a decomposition of $T$.
\begin{enumerate}
	\item[(a)] Write $T$ as a multilinear polynomial $t$ in the ring $K[x_{1,0}, \ldots , x_{1,n_1}; \ldots \ldots ; x_{d,0}, \ldots , x_{d,n_d} ]$.
	\item[(b)] Use \cite{car} to write $t$ in the minimum number of variables. Then, from the proof of Theorem \ref{i1}, it is always possible to write $t$ as a multilinear form in $K[l_{1,0}, l_{1,1}; \ldots ; l_{d,0}, l_{d,1}]$ with $l_{i,0}, l_{i,1}$ linear forms in $K[x_{i,0}, \ldots , x_{i,n_i}]$ for $i=1, \ldots ,d$. 
	\item[(c)] Run Algorithm \ref{algo1}. If Algorithm \ref{algo1} stops at Step (\ref{2}), go to Step (d). If Algorithm \ref{algo1} stops at Step (\ref{3}), then $T\notin \sigma_2(X_{n_1, \ldots , n_d})$. If Algorithm \ref{algo1} stops at Step (\ref{5}), go to Step (e). Otherwise go to Step (f).
\item[(d)] In this case the rank of $T$ is 1, then solve the system $t=m_1(l_{1,0}, l_{1,1})\cdots m_d(l_{d,0}, l_{d,1})$
where $m_i(l_{i,0}, l_{i,1})$ are linear forms in $K[l_{i,0}, l_{i,1}]$, for $i=1, \ldots , d$ (the solution exists and it is unique up to constants).
\item[(e)] In this case the rank of $T$ is 2,  then solve the system 
$t=m_{1,1}(l_{1,0}, l_{1,1})\cdots m_{d,1}(l_{d,0}, l_{d,1})+m_{1,2}(l_{1,0}, l_{1,1}) \cdots  m_{d,2}(l_{d,0}, l_{d,1})$
where
$m_{i,j}(l_{i,0}, l_{i,1})$ are linear forms in $K[l_{i,0}, l_{i,1}]$, for $i=1, \ldots , d$ and $j=1,2$ (the solution exists). 
\item[(f)]  In this case the rank of $T$ is $q/2$ and $t\in K[l_{1,0}, l_{1,1}; \ldots ; l_{q,0}, l_{q,1}]$ for certain $q\leq d$ and there exist $q$ two-dimensional subspaces $W_i\subset V_i$ such that $T$ belongs to the Segre variety $X_{1, \ldots , 1}\subset \mathbb{P}(W_1\otimes \cdots \otimes W_q)$. Let $L_T\subset W_1\otimes \cdots \otimes W_q$ be a generic space of dimension $q$ passing through $T$ and compute a point $O\in X_{1, \ldots , 1}$ such that $[T]\in T_O(X_{1, \ldots , 1})$ (it is sufficient to impose that $\mathbb{P}(L_T)\cap X_{1, \ldots , 1}$ has a double solution). Let $O_1(l_{1,0}, l_{1,1})\cdots n_q(l_{q,0}, l_{q,1})$ and go to Step (g).
\item[(g)] Now it is sufficient to solve the system $t=m_1(l_{1,0}, l_{1,1})\cdots n_q(l_{q,0}, l_{q,1})+ \cdots +n_1(l_{1,0}, l_{1,1})\cdots m_q(l_{q,0}, l_{q,1})$ with 
$m_i(l_{i,0}, l_{i,1})$ linear forms in $K[l_{i,0}, l_{i,1}]$, for $i=1, \ldots , q$ (there exist $\infty^{q/2-1}$ solutions of the system).
\end{enumerate}
\end{algorithm}

\section{On  Comon's conjecture}\label{comon}

In this section we want to relate the result obtained in Theorem \ref{i1} to the Comon's conjecture stated in the Introduction.

Let $\nu_d(\mathbb{P}^n)$ be the classical Veronese embedding of $\mathbb{P}^n$ into $\mathbb{P}^{{n+d\choose d}-1}$ via the sections of the sheaf $\mathcal{O}(d)$. As pointed out in the introduction if $\mathbb{P}^n\simeq \mathbb{P}(V)$ with $V$ an $(n+1)$-dimensional vector space, then $\nu_d(\mathbb{P}^n)\subset \mathbb{P}(S^dV)$ can be interpreted as the variety that parameterizes projective classes of completely decomposable symmetric tensors $T\in S^dV$. Moreover
$$\nu_d(\mathbb{P}^n)=X_{n, \ldots , n}\cap \mathbb{P}(S^dV)\subset \mathbb{P}(V^{\otimes d}).$$

\begin{definition}\label{symrank} Let $P\in \mathbb{P}(S^dV)$ be a projective class of a symmetric tensor. We define the symmetric rank $r_{\nu_d(\mathbb{P}^n)}(P)$ of $P$ as the minimum number of $r$ of points $P_i\in \nu_d(\mathbb{P}^n)$ whose linear span contains $P$.
\end{definition}

With this definition, \textbf{Comon's conjecture} (Conjecture \ref{comonconj}) can be rephrased as follows: 
$$\hbox{if } P\in \mathbb{P}(S^dV) \hbox{ then } r_{\nu_d(\mathbb{P}^n)}(P)=r_{X_{n, \ldots , n}}(P).$$

Obviously $r_{X_{n, \ldots , n}}(P) \leq r_{\nu_d(\mathbb{P}^n)}(P)$. In \cite{cglm} the authors prove the reverse inequality for a general $d$-tensor ($d$ even and large)
with rank at most $n$ (Proposition 5.3) 
 and for $ r_{X_{n, \ldots  ,n}}(P)=1,2$. 

With Theorem \ref{i1} we can prove that conjecture for all symmetric tensors of border rank 2.

\begin{corollary}\label{comond} Let $P\in \sigma_2(\nu_d(\mathbb{P}^n))$. Then $r_{\nu_d(\mathbb{P}^n)}(P)=r_{X_{n, \ldots , n}}(P).$
\end{corollary}

\begin{proof} For any projective variety $X$ we can observe that $\sigma_2(X)=X\cup \tau(X) \cup \sigma_2^0(X)$. 
\\
If $P\in \nu_d(\mathbb{P}^n)\subset X_{n, \ldots , n}$ then there exist $v\in V$ such that $P=[v^{\otimes d} ]\in \nu_d(\mathbb{P}^n)\subset X_{n, \ldots , n}$, therefore obviously $r_{X_{n, \ldots , n}}(P)=r_{\nu_d(\mathbb{P}^n)}(P)=1$.

If $P\in \sigma_2^0(\nu_d(\mathbb{P}^n))$ then $r_{\nu_d(\mathbb{P}^n)}(P)=2$, that implies that $r_{X_{n, \ldots , n}}(P)\leq 2$, and therefore by \cite{cglm}, that we have that $r_{X_{n, \ldots , n}}(P)=r_{\nu_d(\mathbb{P}^n)}(P)=2$.

Now assume that $P\in \tau(\nu_d(\mathbb{P}^n))\setminus \nu_d(\mathbb{P}^n)$ and that $\sigma_2(\nu_d(\mathbb{P}^n))\neq \tau(\nu_d(\mathbb{P}^n))$. For such a $P$ we know that $r_{\nu_d(\mathbb{P}^n)}(P)=d$ (see \cite{Sy}, \cite{CS}, \cite{bcmt}, \cite{bgi}). Any point $P\in \tau(\nu_d(\mathbb{P}^n))\setminus \nu_d(\mathbb{P}^n)$ can be thought as the projective class of a homogeneous degree $d$ polynomial in $n+1$ variables for which there exist two linear forms $L,M$ in $n+1$ variables such that $P=[L^{d-1}M]$; hence $d$ is the minimum integer $k$ such that $P\in \langle\nu_k(\mathbb{P}^n) \rangle$. Therefore $\eta_{X_{n, \ldots , n}}(P)=d$. Since obviously $\tau(\nu_d(\mathbb{P}^n))\subset \tau(X_{n, \ldots , n})$ we have that, by Theorem \ref{i1}, $r_{X_{n, \ldots , n}}(P)=\eta_{X_{n, \ldots , n}}(P)$.
\end{proof}

\providecommand{\bysame}{\leavevmode\hbox to3em{\hrulefill}\thinspace}

\end{document}